\documentclass[final, nomarks]{dmtcs-episciences}

\usepackage[utf8]{inputenc}
\usepackage{subfigure}
\RequirePackage{amsmath,amsthm,amssymb}
\usepackage{tikz}
\usepackage{pdflscape}
\usepackage{biblatex}

\newtheorem{thm}{Theorem}
\newtheorem{lem}[thm]{Lemma}
\newtheorem{prop}[thm]{Proposition}
\newtheorem{corol}[thm]{Corollary}

\newcommand{\I}{\mathcal{I}}
\renewcommand{\min}{\mathbf{min}}
\renewcommand{\max}{\mathbf{max}}
\newcommand{\Zeros}{\mathbf{Zeros}}
\newcommand{\child}{\mathbf{child}}
\newcommand{\Sites}{\mathbf{Sites}}
\newcommand{\Vals}{\mathbf{Vals}}
\newcommand{\push}{\mathbf{push}}
\newcommand{\gray}{\textcolor{gray}}
\newcommand{\Ustirling}[2]{\genfrac{[}{]}{0pt}{}{#1}{#2}}
\let\oldqedhere\qedhere
\renewcommand{\qedhere}{\pushQED{\qed}\oldqedhere}

\author[Benjamin Testart]{Benjamin Testart\affiliationmark{1}}
\title{Generating trees growing on the left for pattern-avoiding inversion sequences}

\affiliation{Université de Lorraine, CNRS, Inria, LORIA, F-54000 Nancy, France}
\keywords{pattern avoidance, inversion sequences, enumeration, generating trees, generating functions}

\begin{document}
\publicationdata{vol. 27:1, Permutation Patterns 2024}{2025}{10}{10.46298/dmtcs.14716}{2024-11-11; 2024-11-11; 2025-08-12}{2025-09-04}
\maketitle
\begin{abstract} \medskip
This work concerns a construction of pattern-avoiding inversion sequences from right to left we call the \emph{generating tree growing on the left}. We first apply this construction to inversion sequences avoiding 201 and 210, resulting in a new way of computing their generating function. We then use a slightly modified construction to compute the generating function of inversion sequences avoiding 010 and 102, which was only conjectured before. These generating functions are algebraic in both instances. We end by discussing how the generating tree growing on the left can be applied in a more general setting.
\end{abstract}

\def\endproof{\medskip}
\renewcommand{\thefootnote}{\arabic{footnote}}
\section{Introduction}

\subsection{Basic definitions}
For all $n \geqslant 0$, let $\I_n$ be the set of \emph{inversion sequences} \cite{Mansour_Shattuck_2015, Corteel_Martinez_Savage_Weselcouch_2016} of size $n$, that is, the set of sequences $\sigma = (\sigma_1, \dots, \sigma_n)$ such that $\sigma_i \in \{0, \dots, i-1\}$ for all $i \in \{1, \dots, n\}$. Let $\I = \coprod_{n \geqslant 0} \I_n$ be the set of all inversion sequences.

We define a \emph{pattern} to be a finite, nonempty sequence of nonnegative integers $\rho$ whose set of values forms an interval of integers starting at 0. In this work, we exclusively study patterns of length 3, which we denote $\rho_1\rho_2\rho_3$ instead of $(\rho_1,\rho_2,\rho_3)$. For instance, $102$ and $010$ are patterns, but $022$ and $123$ are not.

We say that an inversion sequence $\sigma \in \I_n$ \emph{contains} a pattern $\rho$ of length $k \leqslant n$ if there exists a subsequence of $k$ entries of $\sigma$ whose values appear in the same order as the values of $\rho$. Such a subsequence is called an \emph{occurrence} of $\rho$. An inversion sequence \emph{avoids} the pattern $\rho$ if it does not contain $\rho$. For instance, the sequence $(0,1,0,3,0,2,3)$ avoids the pattern 110, but it contains the pattern 021 since the subsequence $(0,3,2)$ is an occurrence of 021. If $P$ is a set of patterns, we denote by $\I(P)$ the set of inversion sequences avoiding all patterns in $P$, and by $\I_n(P)$ its subset of sequences of size $n$.

\subsection{Context and summary of results}
Pattern-avoiding inversion sequences were introduced independently in \cite{Mansour_Shattuck_2015} and \cite{Corteel_Martinez_Savage_Weselcouch_2016}, by analogy with pattern-avoiding permutations. Afterwards, many more articles studied inversion sequences avoiding patterns of length three, see \cite{Martinez_Savage_2018, Beaton_Bouvel_Guerrini_Rinaldi_2019, Yan_Lin_2020, Kotsireas_Mansour_Yildirim_2024, Chen_Lin, Pantone201+210, Testart2024}, among others.

Our work is concerned with two distinct classes of pattern-avoiding inversion sequences: $\I(201,210)$ and $\I(010,102)$.
The class $\I(201,210)$ is counted by sequence A212198 in \cite{OEIS}. It has an algebraic generating function, which was found in both \cite{Chen_Lin} and \cite{Pantone201+210}, using two different constructions.
Chen and Lin \cite{Chen_Lin} use a decomposition around the rightmost saturated\footnote{In an inversion sequence $\sigma$, the entry $\sigma_i$ is said to be ``saturated" if $\sigma_i = i-1$ (since $i-1$ is the largest value which could appear at position $i$).} entry of inversion sequences, with one parameter counting the number of saturated entries.
The construction used by Pantone \cite{Pantone201+210} is a generating tree in one integer parameter and two boolean parameters, and it involves some ``commitments" to future decisions.

The class $\I(010,102)$ is counted by sequence A374553 in \cite{OEIS}. This sequence was computed in polynomial time in \cite{Testart2024} using a decomposition involving $\{010,102\}$-avoiding words, and it was conjectured to have an algebraic generating function in \cite{Pantone201+210} using the software \cite{guessfunc}.

In Section \ref{sectionGT}, we describe a construction for inversion sequences, which we call the ``generating tree growing on the left". This construction first appeared in \cite{Corteel_Martinez_Savage_Weselcouch_2016} to enumerate the classes $\I(101)$ and $\I(110)$. A modified construction based on this one also appears in \cite{Lin_Fu_2021} to enumerate inversion sequences avoiding the vincular pattern $\underline{12}0$. 

As it is often the case with generating trees, our general approach in this paper is to first look for a simple succession rule describing the tree (up to isomorphism), then turn this rule into a functional equation for a multivariate generating function, then solve this equation with the kernel method \cite{GFforGT}.

In Section \ref{201+210}, we consider the generating tree growing on the left for inversion sequences avoiding 201 and 210. Our succession rule corresponds to a functional equation which contains two evaluations $A(t,1)$ and $A(t,2)$ of a bivariate series $A$. This is quite unusual, and it occurs because the succession rule produces labels with multiplicities that are powers of 2. Regardless, this equation can still be solved with the kernel method, resulting in a new and simpler way of computing the generating function of $\I(201,210)$.

In Section \ref{010+102}, we modify the generating tree growing on the left construction to apply it to inversion sequences avoiding 010 and 102. In this case, we obtain a functional equation involving two catalytic variables, which can be solved by simply applying the kernel method twice. This allows us to compute the generating function of $\I(010,102)$, confirming a conjecture of \cite{Pantone201+210}.

In Section \ref{secConclusion}, we briefly explore how widely the generating tree growing on the left could be applied to other similar problems. To that end, we present some succession rules for inversion sequences avoiding any pattern of length 3.

\subsection{Notation}
For any integers $a,b \in \mathbb Z$, we denote the integer interval $[a,b] = \{k \in \mathbb Z \; : \; a \leqslant k \leqslant b\}$.
We denote by $\varepsilon$ the \emph{empty sequence}, that is, the only sequence of size $0$.
We define $\chi$ to be the characteristic function of a statement, that is $\chi(\text{S}) = \begin{cases}
    0 & \text{if \; S is false}, \\
    1 & \text{if \; S is true}.
\end{cases}$

For all $n \geqslant 0$ and $\sigma \in \I_n$, let
\begin{itemize}
    \item $|\sigma| = n$ be the \emph{size} of $\sigma$,
    \item $\Vals(\sigma) = \{\sigma_i \; : \; i \in [1,n]\}$ be the \emph{set of values} of $\sigma$,
    \item $\min(\sigma) = \min(\Vals(\sigma))$ be the \emph{minimum} of $\sigma$, with the convention $\min(\varepsilon) = +\infty$,
    \item $\max(\sigma) = \max(\Vals(\sigma))$ be the \emph{maximum} of $\sigma$, with the convention $\max(\varepsilon) = -1$,
    \item $\Zeros(\sigma) = \{i \in [1,n] \, : \, \sigma_i = 0\}$ be the set of (positions of) zeros of $\sigma$.
\end{itemize}

\section{Generating trees for inversion sequences growing on the left} \label{sectionGT}
A \emph{generating tree} \cite{West1, West2} is a rooted, labelled tree such that the label of each node determines its number of children, and their labels. A generating tree can be defined by a \emph{succession rule}
$$\Omega = \left \{ \begin{array}{rclll} 
    a \\
    \ell & \leadsto & i^{m(\ell,i)} & \text{for} & i \in C(\ell) \\
\end{array} \right.$$
composed of an \emph{axiom} $a$, which labels the root of the tree, and a \emph{production} which associates to each label $\ell$ the labels of the children of any node labelled $\ell$ in the tree. In the rule $\Omega$ above, each node labelled $\ell$ has $m(\ell,i)$ children labelled $i$, for each $i$ in a finite set of labels $C(\ell)$.

Generating trees are used to solve enumeration problems by identifying combinatorial objects with the nodes of a tree, so that each object of size $n$ corresponds to a node at level $n$ in the tree. This is useful when many nodes are roots of isomorphic subtrees, as the combinatorial objects corresponding to these nodes may all be given the same label.
A succession rule for a generating tree then induces a recurrence relation on the number of objects of each size (which usually involves some parameters corresponding to the labels of the tree), or equivalently a functional equation for a (multivariate) generating function.
For a more thorough introduction to generating trees in the context of pattern-avoiding inversion sequences, see \cite[Section 2.1]{Testart2024}.

We define the generating tree for inversion sequences \emph{growing on the left} as the only generating tree which has one node labelled by each inversion sequence (we now identify the nodes with their labels), whose root is the empty sequence $\varepsilon$, and in which the parent of each nonempty sequence is obtained by deleting its leftmost entry (which is always a zero), and subtracting 1 from all its positive entries.
In this tree, the children of an inversion sequence are obtained by
\begin{enumerate}
    \item first adding 1 to every positive entry in the sequence,
    \item then adding 1 to some (or none) of the zero entries,
    \item lastly, inserting a zero at the leftmost position.
\end{enumerate}
Observe that this construction is indeed a generating tree (it produces only inversion sequences and, starting from the empty sequence $\varepsilon$, it produces every inversion sequence exactly once). Note that steps 1 and 3 are deterministic, whereas step 2 involves choosing a subset of the zero entries. For any inversion sequence $\sigma \in \I$, and any $Z \subseteq \Zeros(\sigma)$, we denote
$$\child(\sigma, Z) = 0 \cdot (\sigma_i + \chi(\sigma_i > 0) + \chi(i \in Z))_{i \in [1,|\sigma|]}$$
the child of $\sigma$ corresponding to the subset of zeros $Z$ to which 1 is added. It should be apparent that $\sigma$ has $2^{|\Zeros(\sigma)|}$ children.
For instance, the sequence $\sigma = (0,1,1,3,0,2,5)$ satisfies 
$\Zeros(\sigma) = \{1,5\}$, and its four children are
\begin{align*}
    \child(\sigma, \emptyset) &= (0,0,2,2,4,0,3,6), \\
    \child(\sigma, \{1\}) &= (0,1,2,2,4,0,3,6), \\
    \child(\sigma, \{5\}) &= (0,0,2,2,4,1,3,6), \\
    \child(\sigma, \{1,5\}) &= (0,1,2,2,4,1,3,6).
\end{align*}
Up to a change of labels, the generating tree for inversion sequences growing on the left is described by the succession rule $\Omega_{\text{Left}}$ below, where the new label $k$ corresponds to the number of zeros in a sequence. This generating tree is represented with both labellings on Figure \ref{GT_left}.
$$\Omega_{\text{Left}} = \left \{ \begin{array}{rclll} 
    (0) \\
    (k) & \leadsto & (i+1)^{\binom{k}{i}} & \text{for} & i \in [0,k] \\
\end{array} \right.$$

\begin{figure}[ht]
\centering
\begin{tikzpicture}[level distance=25mm]
\tikzstyle{level 2}=[sibling distance=65mm]
\tikzstyle{level 3}=[sibling distance=20mm]
\node[align=center]{\large $\varepsilon$ \\ $(0)$}
child{node[align=center]{0 \\ $(1)$}   
    child{node[align=center]{00 \\ $(2)$}
        child{node[align=center]{000 \\ $(3)$}}
        child{node[align=center]{010 \\ $(2)$}}
        child{node[align=center]{001 \\ $(2)$}}
        child{node[align=center]{011 \\ $(1)$}}
    }
    child{node[align=center]{01 \\ $(1)$}
        child{node[align=center]{002 \\ $(2)$}}
        child{node[align=center]{012 \\ $(1)$}}
    }
};
\end{tikzpicture}
\caption{First four levels of the generating tree growing on the left for all inversion sequences.}  
\label{GT_left}
\end{figure}

The simplest construction of inversion sequences is the generating tree growing sequences on the right, used in e.g. \cite{Kotsireas_Mansour_Yildirim_2024} or \cite[Section 2]{Testart2024}.
In comparison, growing sequences on the left is not the most practical way to construct the set of inversion sequences (it is not even obvious that the tree described by $\Omega_{\text{Left}}$ has $n!$ nodes at each level $n$).
However, this construction becomes simpler for some classes of pattern-avoiding inversion sequences. For instance, \cite{Corteel_Martinez_Savage_Weselcouch_2016} proves that the generating trees growing on the left for the classes $\I(101)$ and $\I(110)$ are both described by the succession rule
$$\Omega_{\text{pCat}} = \left \{ \begin{array}{rclll}
    (0) \\
    (k) & \leadsto & (k+1) \\
    && (i)^{i} & \text{for} & i \in [1,k]. \\
\end{array} \right.$$

The generating tree growing on the left for inversion sequences lacks a property that is typically desired when working with objects that avoid patterns: it does not automatically restrict to a generating tree for families of pattern-avoiding inversion sequences. Specifically, restricting the tree to keep only sequences which avoid some patterns may disconnect the tree. This occurs because an inversion sequence which contains a pattern $\rho$ may have children which avoid $\rho$. For instance, the sequence $(0,0,2,1)$ avoids the pattern 010, but its parent sequence is $(0,1,0)$.
We now characterize the patterns for which this restriction can be done, in Propositions \ref{PropNecessaryPatterns} and \ref{PropSufficientPatterns}. First, we show that this restriction is possible when avoiding a pattern which contains only one zero, or two zeros that are both at the beginning of the pattern.
\begin{prop} \label{PropNecessaryPatterns}
    Let $\rho$ be a pattern of length $k \geqslant 1$ which either contains a single zero, or contains exactly two zeros and begins with 00. Let $\sigma$ be an inversion sequence which contains an occurrence $(\sigma_{q_i})_{i \in [1,k]}$ of $\rho$, and let $\tau$ be a child of $\sigma$ in the generating tree growing on the left. Then $\tau$ contains the pattern $\rho$.
\end{prop} 
\begin{proof}
     If $\rho$ contains a single zero, then $(\sigma_{q_i})_{i \in [1,k]}$ contains at most one zero. It follows that $(\sigma_{q_i})_{i \in [1,k]}$ is order-isomorphic to $(\tau_{q_i+1})_{i \in [1,k]}$ (every nonzero value is increased by 1, and the hypothetical zero can either increase by 1 or remain a zero), hence $\tau$ contains the pattern $\rho$.

    If $\rho$ contains exactly two zeros and begins with 00, we distinguish some cases depending on the values of $(\sigma_{q_i})_{i \in [1,k]}$.
     \begin{itemize}
         \item If $(\sigma_{q_i})_{i \in [1,k]}$ contains only positive values, then $(\tau_{q_i + 1})_{i \in [1,k]} = (\sigma_{q_i}+1)_{i \in [1,k]}$ is still an occurrence of $\rho$, in $\tau$.
         \item If $(\sigma_{q_i})_{i \in [1,k]}$ contains zero values, i.e. $\sigma_{q_1} = \sigma_{q_2} = 0$, then $\tau_{q_1+1}, \tau_{q_2+1} \in \{0,1\}$.
         \begin{itemize}
             \item If $(\tau_{q_1+1}, \tau_{q_2+1}) \in \{(0,0), (1,1)\}$, then $(\tau_{q_i + 1})_{i \in [1,k]}$ is still an occurrence of $\rho$.
             \item If $(\tau_{q_1+1}, \tau_{q_2+1}) = (0,1)$, then $(\tau_1) \cdot (\tau_{q_i + 1})_{i \in \{1\} \cup [3,k]}$ is an occurrence of $\rho$.
             \item If $(\tau_{q_1+1}, \tau_{q_2+1}) = (1,0)$, then $(\tau_1) \cdot (\tau_{q_i + 1})_{i \in [2,k]}$ is an occurrence of $\rho$. \qedhere
         \end{itemize}
     \end{itemize}
\end{proof}

Proposition \ref{PropNecessaryPatterns} provides sufficient conditions on a pattern $\rho$ to ensure that the complete generating tree growing on the left restricts to a generating tree for the class $\I(\rho)$. We now show that these conditions are also necessary.

\begin{prop} \label{PropSufficientPatterns}
    Let $\rho$ be a pattern of length $k \geqslant 2$ which contains at least two zeros, and either begins with 000, or has a zero to the right of a positive value. Then there exist two inversion sequences $\sigma$ and $\tau$ such that
    \begin{itemize}
        \item $\tau$ is a child of $\sigma$ in the generating tree growing on the left,
        \item $\sigma$ contains $\rho$,
        \item $\tau$ avoids $\rho$.
    \end{itemize}
\end{prop}
\begin{proof}
    If $\rho$ begins with 000, let $\rho' = (\rho_i + k \cdot \chi(\rho_i > 0))_{i \in [3,k]}$ be the sequence obtained by deleting the first two zeros of $\rho$ and adding $k$ to every nonzero value. Let $\sigma = (0,0,1,1, \dots , k-1,k-1) \cdot \rho'$. It is easily checked that $\sigma$ is an inversion sequence which contains the pattern $\rho$. Let $\tau = \child(\sigma, \{1, 2\})$. By construction, $\tau$ has one fewer zero than $\rho$, hence no occurrence of $\rho$ in $\tau$ can contain the value 0. Similarly, every value in $[1,k]$ appears exactly twice in $\tau$, while the value 0 appears at least three times in $\rho$, so any occurrence of $\tau$ in $\rho$ must only contain values greater than $k$. Only the rightmost $k-2$ entries of $\tau$ may have a value greater than $k$; this is not enough to form an occurrence of $\rho$.

    If $\rho$ has a zero to the right of a positive value, let $\sigma = (0)^k \cdot \rho$. It is easy to see that $\sigma$ is an inversion sequence which contains  $\rho$. Let $\tau = \child(\sigma, \{\max(\Zeros(\sigma))\})$. By hypothesis, the rightmost zero of $\sigma$ (or equivalently, $\rho)$ is necessarily an inversion bottom, i.e. there exists an entry of greater value to its left.
    This implies that the number of zeros that are inversion bottoms in $\tau$ is one fewer than in $\rho$, hence no occurrence of $\rho$ in $\tau$ can contain the value 0. Additionally, the value 1 appears only once in $\tau$, while the value 0 appears at least twice in $\rho$, so any occurrence of $\rho$ in $\tau$ must only contain values greater than 1. There are at most $k-2$ entries of $\tau$ having value at least 2; this is not enough to form an occurrence of $\rho$. \qedhere
\end{proof}

By Proposition \ref{PropNecessaryPatterns}, the generating tree growing on the left is well-defined for inversion sequences avoiding the patterns 201 and 210 studied in Section \ref{201+210}, since both patterns contain a single zero. 
Even for problematic patterns like 010, the problem can be worked around by modifying our construction, in some cases. For instance, we describe a generating tree for the class $\I(010,102)$ in Section \ref{010+102}. Some other workarounds for the classes $\I(000)$, $\I(010)$, and $\I(100)$ can be found in Section \ref{secConclusion}.

\section{Inversion sequences avoiding 201 and 210} \label{201+210}

For all $n \geqslant 0$ and $\sigma \in \I_n(201,210)$, let $L(\sigma) = \{i \in [1,n] \, : \, \forall j \leqslant i, \; \sigma_j = 0\}$ be the set of leading zeros of $\sigma$, and let $R(\sigma) = \Zeros(\sigma) \backslash L(\sigma)$. Let $\ell(\sigma) = |L(\sigma)|$ be the number of leading zeros of $\sigma$. If $\ell(\sigma) < n$ (i.e. if $\sigma$ is not constant), then $\sigma_{\ell(\sigma)+1}$ is the value of the first nonzero entry of $\sigma$. Notice that $L(\sigma)$ (resp. $R(\sigma)$) is the set of zeros of $\sigma$ to the left (resp. right) of position $\ell(\sigma)+1$.

\begin{prop} \label{propZeros201+210}
    Let $\sigma \in \I_n(201,210)$, $Z \subseteq \Zeros(\sigma)$, and let $\tau = \child(\sigma, Z) \in \I_{n+1}$. Then $\tau$ avoids 201 and 210 if and only if $Z \cap R(\sigma) \in \{\emptyset, R(\sigma)\}$.
\end{prop}
\begin{proof}
    Since the patterns 201 and 210 both begin with the value 2, the first value of an occurrence of either pattern in $\tau$ must be at least 2. For all $i \in [1,\ell(\sigma)+1]$, we have $\tau_i \leqslant 1$, therefore any occurrence of 201 or 210 in $\tau$ must be in $(\tau_i)_{i \in [\ell(\sigma)+2, n+1]}$. If $Z \cap R(\sigma) \in \{\emptyset, R(\sigma)\}$, then $(\tau_i)_{i \in [\ell(\sigma)+2, n+1]}$ is order-isomorphic to $(\sigma_i)_{i \in [\ell(\sigma)+1, n]}$, which avoids both patterns.

    Assume $Z \cap R(\sigma) \notin \{\emptyset, R(\sigma)\}$. Let $i,j \in R(\sigma)$ be such that $i \in Z$ and $j \notin Z$. In particular, $\tau_{\ell(\sigma)+2} \geqslant 2$, $\tau_{i+1} = 1$, and $\tau_{j+1} = 0$. If $i < j$, then $(\tau_{\ell(\sigma)+2}, \tau_{i+1}, \tau_{j+1})$ is an occurrence of 210. If $j < i$, then $(\tau_{\ell(\sigma)+2}, \tau_{j+1}, \tau_{i+1})$ is an occurrence of 201. \qedhere
\end{proof}

The proposition above establishes that in the generating tree growing on the left for $\I(201,210)$, the children of $\sigma$ are the sequences such that the subset $Z$ of zeros to which the value 1 is added must contain either all or none\footnote{When $R(\sigma)$ is not empty, one can think of it as containing only a single zero, since all zeros in this set are modified in the same way.} of $R(\sigma)$. Any subset of $L(\sigma)$ may intersect with $Z$ without creating an occurrence of 201 or 210. It follows that the number of children of $\sigma$ is $2^{\ell(\sigma) + \chi(R(\sigma) \neq \emptyset)}$.  

For all $\sigma \in \I(201,210)$, let $s(\sigma) = \ell(\sigma) + \chi(R(\sigma) \neq \emptyset)$.

\begin{lem}
The generating tree growing on the left for $\I(201,210)$ is described by the succession rule
$$\Omega_{\{201,210\}} = \left \{ \begin{array}{rclll}
    (0) \\
    (s) & \leadsto & (i) & \text{for} & i \in [1,s+1] \\
    && (i+1)^{2^{s-i}-1} & \text{for} & i \in [1,s-1].
\end{array} \right.$$
\end{lem}
\begin{proof}
    We show that the succession rule $\Omega_{\{201,210\}}$ is obtained by replacing each sequence $\sigma$ by the label $s(\sigma)$ in the generating tree growing on the left for $\I(201,210)$. First, the sequence $\varepsilon$ corresponding to the root has label $s(\varepsilon) = 0$.

    Let $\sigma \in \I(201,210)$. We know from Proposition \ref{propZeros201+210} that the set of children of $\sigma$ in the generating tree growing on the left for $\I(201,210)$ is
    $$\{\child(\sigma, Z) \; : \; Z \subseteq L(\sigma)\} \cup \{\child(\sigma, Z' \sqcup R(\sigma)) \; : \; Z' \subseteq L(\sigma)\}.$$
    Note that the above union is disjoint if $R(\sigma)$ is nonempty, and both sets in the union are equal if $R(\sigma)$ is empty.
    Let $Z$ be such that $\tau = \child(\sigma, Z) \in \I(201,210)$ (i.e. $Z \subseteq L(\sigma)$, or $Z = Z' \sqcup R(\sigma)$ for some $Z' \subseteq L(\sigma)$).
    We distinguish several cases:
\begin{itemize}
    \item If $R(\sigma) = \emptyset$, then $s(\sigma) = \ell(\sigma)$.
    \begin{itemize}
        \item If $Z$ is a (possibly empty) interval of positions at the end of $L(\sigma)$, that is if $Z = [i,\ell(\sigma)]$ for some $i \in [1,\ell(\sigma)+1]$ (the case $i=\ell(\sigma)+1$ corresponding to $Z = \emptyset$), then $L(\tau) = [1,i]$ and $R(\tau) = \emptyset$, therefore $s(\tau) = i$.
        \item Otherwise, let $i = \min(Z)$. It holds that $i = \ell(\tau) \in [1, \ell(\sigma)-1]$, and there exists $j \in \Zeros(\sigma) \backslash Z$ such that $j > i$. 
        In particular, $R(\tau) \neq \emptyset$, and $s(\tau)  = i+1$. There are $2^{s(\sigma)-i}-1$ possible sets $Z$ satisfying the present assumptions (that is all sets of the form $Z = \{i\} \sqcup Z'$ for some $Z' \subsetneq [i+1, \ell(\sigma)]$).
    \end{itemize}
    \item If $R(\sigma) \neq \emptyset$, then $s(\sigma) = \ell(\sigma)+1$.
    \begin{itemize}
        \item If $Z = \emptyset$, then $L(\tau) = [1,\ell(\sigma)+1]$ and $R(\tau) = \{i+1 \; | \; i \in R(\sigma)\}$, therefore $s(\tau) = s(\sigma) + 1$.
        \item If $Z = \Zeros(\sigma) \cap [i,n]$ for some $i \in [1,\ell(\sigma)+1]$, then $L(\tau) = [1,i]$ and $R(\tau) = \emptyset$, therefore $s(\tau) = i$.
        \item Otherwise, $Z$ is not empty nor consists of a set of final zeros of $\sigma$. This implies as before that $R(\tau) \neq \emptyset$, and that $s(\tau) = i+1$, where $i = \ell(\tau) = \min(Z) \in [1,\ell(\sigma)]$. There are $2^{s(\sigma)-i}-1$ possible sets $Z$ satisfying the present assumptions: either $Z = \{i\} \sqcup Z'$ for some $Z' \subseteq [i+1, \ell(\sigma)]$, or $Z = \{i\} \sqcup Z' \sqcup R(\sigma)$ for some $Z' \subsetneq [i+1, \ell(\sigma)]$). \qedhere
    \end{itemize}
\end{itemize}
\end{proof}

This succession rule gives a new, simple way of computing the generating function of $\I(201,210)$. For all $n,k \geqslant 0$, let $\mathfrak a_{n,k} = |\{\sigma \in \I_n(201,210) \; : \; k = s(\sigma)\}|$. Let $A(t,u) = \sum_{n,k \geqslant 0} \mathfrak a_{n,k} t^n u^k$.
\begin{thm}
    The generating function of $\I(201,210)$ is $A(t,1) = \frac{2-t-t\sqrt{1-8t}}{4t^2 - 4t + 2}$.
\end{thm}
\begin{proof}
The succession rule $\Omega_{\{201,210\}}$ gives the following equation:
\begin{align*}
    A(t,u) &= 1 + \sum_{n,k \geqslant 0} \mathfrak a_{n,k} t^{n+1} \left ( \sum_{i=1}^{k+1} u^i + \sum_{i=1}^{k-1} (2^{k-i} - 1) u^{i+1} \right ) \\
    &= 1 + \sum_{n,k \geqslant 0} \mathfrak a_{n,k} t^{n+1} \left (u + \sum_{i=1}^{k} 2^{k-i} u^{i+1} \right) \\
    &= 1 + \sum_{n,k \geqslant 0} \mathfrak a_{n,k} t^{n+1} \left (u + 2^k u \sum_{i=1}^{k} \left (\frac{u}{2} \right )^i \right) \\
    &= 1 + \sum_{n,k \geqslant 0} \mathfrak a_{n,k} t^{n+1} \left (u + u^2 \frac{2^k - u^k}{2-u} \right) \\
    &= 1 + tu A(t,1) + \frac{tu^2}{2-u} (A(t,2) - A(t,u)).
\end{align*}
This can be rewritten as
\begin{equation} \label{eqA}
   \left (1+\frac{tu^2}{2-u} \right ) A(t,u) = \frac{tu^2}{2-u} A(t,2) + tu A(t,1) + 1.
\end{equation}
We apply the kernel method. The kernel in \eqref{eqA} is $1+\frac{tu^2}{2-u}$. To cancel the kernel, we substitute $u$ for $U(t) = \frac{1-\sqrt{1-8t}}{2t}$. This yields the equation
\begin{equation} \label{eqAkernel}
    A(t,2) = tU(t)A(t,1)+1.
\end{equation}
Evaluating \eqref{eqA} for $u=1$ gives
$$A(t,1) = tA(t,2)+1.$$
Replacing $A(t,2)$ by its expression from \eqref{eqAkernel}, we obtain
$$A(t,1) = t^2 U(t) A(t,1) + t + 1.$$
This gives an expression for $A(t,1)$:
\begin{align*}
    A(t,1) &= \frac{t+1}{1-t^2 U(t)} = \frac{2-t-t\sqrt{1-8t}}{4t^2 - 4t + 2}. \qedhere
\end{align*}
\end{proof}

\section{Inversion sequences avoiding 010 and 102} \label{010+102}

We begin with two simple but useful properties satisfied by inversion sequences avoiding the pattern 010.

\begin{prop} \label{prop010}
    Let $\sigma \in \I(010)$. Then
    \begin{enumerate}
        \item all entries of value 0 in $\sigma$ must be consecutive and at the beginning of $\sigma$,
        \item all entries of value 1 in $\sigma$ must be consecutive.
    \end{enumerate}
\end{prop}
\begin{proof}
    \begin{enumerate}
        \item Assume for the sake of contradiction that $\sigma$ contains two non-consecutive entries of value 0. In other words, there exists a subsequence $(\sigma_i, \sigma_j, \sigma_k)$ such that $i<j<k$, $\sigma_i = \sigma_k = 0$, and $\sigma_j \neq 0$. In particular, $\sigma_j \geqslant 1$, therefore $(\sigma_i, \sigma_j, \sigma_k)$ is an occurrence of the pattern 010.
        
        We showed that all entries of value 0 in $\sigma$ must be consecutive. Since any nonempty inversion sequence $\sigma$ satisfies $\sigma_1 = 0$, the consecutive entries of value 0 must be at the beginning of $\sigma$.
        \item Assume for the sake of contradiction that $\sigma$ contains two non-consecutive entries of value 1. In other words, there exists a subsequence $(\sigma_i, \sigma_j, \sigma_k)$ such that $i<j<k$, $\sigma_i = \sigma_k = 1$, and $\sigma_j \neq 1$. We showed that all entries of value 0 are consecutive and at the beginning of $\sigma$, hence $\sigma_j \neq 0$. It follows that $\sigma_j \geqslant 2$, therefore $(\sigma_i, \sigma_j, \sigma_k)$ is an occurrence of the pattern 010. \qedhere
    \end{enumerate}
\end{proof}

\subsection{Modifying the construction}

As we showed in Proposition \ref{PropSufficientPatterns}, the generating tree growing on the left is not always well-defined for inversion sequences avoiding a pattern with several zeros. Specifically, the children of sequences of $\I_n(010,102)$ in this tree are precisely the sequences of $\I_{n+1}(010,102)$ in which the value 1 does not appear to the right of a greater value (otherwise their parent sequence would contain a 0 to the right of a greater value, and therefore contain the pattern 010).
To complete this construction, we shall introduce some new steps so that the value 1 may appear at any position (while avoiding the patterns 010 and 102).

A possible solution could be to modify the way that entries of value 1 are placed, allowing them to be a newly inserted element in the sequence rather than a replacement of a zero entry.
However, we remark that there are some unwanted restrictions on inversion sequences that can be generated by inserting a new 1 in this manner. Indeed, if a sequence $\tau$ were obtained by inserting a new entry of value 1 at position $i$ in an inversion sequence $\sigma \in \I_n$, then for all $j \in [i+1, n+1]$, we would have $\tau_j = \sigma_{j-1} < j-1$. In particular, a sequence $\tau$ which contains a saturated entry $\tau_j = j-1$ for some $j > i$ cannot be generated this way. Nevertheless, for the class $\I(010,102)$, we specifically want to generate a sequence $\tau$ which avoids 102, and in which the value 1 inserted at position $i$ is to the right of a value $k > 1$. In particular, for all $j > i$, $\tau$ must satisfy $\tau_j \leqslant k < i-1 < j-1$ to avoid 102, so the aforementioned problem disappears. Do note that inserting an entry of value 1 to the left of every greater value would still cause this problem, so we cannot place every single entry of value 1 by insertion either. This leads us to consider a hybrid approach in which entries of value 1 are only placed by insertion if they are to the right of a greater value, and still placed by replacing a zero otherwise.
This describes a new generating tree for $\I(010,102)$, in which the parent of each nonempty sequence is obtained by
\begin{itemize}
    \item removing an entry of value 1, if such an entry appears to the right of a greater value (it does not matter which entry of value 1 is removed since they are all consecutive by Proposition \ref{prop010});
    \item removing the first zero and subtracting 1 from all nonzero entries (as in Section \ref{sectionGT}), otherwise.
\end{itemize}

We introduce a second small change in this construction. It is difficult to explain precisely the reason for this choice, but it allows for a simpler presentation of the computation of the generating function of $\I(010, 102)$ in Theorem \ref{thmGF010_102}, by reducing the number of cases to consider.
If a sequence contains an entry of value 1 with no greater value to its left (i.e. the sequence begins with a prefix of zeros, followed by a 1), then its parent is now obtained by subtracting 1 from the leftmost entry of value 1 only\footnote{When exploring the ancestors of a sequence of this form, its parent in our original definition will eventually appear, after all entries of value 1 have been turned into zeros. We are simply placing some intermediate sequences on that path.}.
Otherwise, its parent is determined as described just above.
The corresponding tree is represented in Figure \ref{GT010+102}. It should be noted that this tree does not quite conform to the usual definition of a generating tree, as some nodes are on the same level as their parent.

The article \cite{JumpingGT} introduces the notion of ``jumping" generating trees, in which each edge has a (bounded) positive integer length.
If a node at level $n$ has a child through an edge of length $k$, then this child is at level $n+k$ (in a classical generating tree, all edges have length $1$). The tree in Figure \ref{GT010+102} could be seen as a jumping generating tree having edges of length 0 or 1.

We claim that the following succession rule describes a generating tree isomorphic to the one described above (and represented in Figure \ref{GT010+102}); the arrow $\overset{k} \leadsto$ marks an edge of length $k$.
$$\Omega_{\{010,102\}} = \left \{ \begin{array}{rclll}
    (\mathtt C,0) \\
    (\mathtt C,z) & \overset{1} \leadsto & (\mathtt C,z+1) \\
    & \overset{0} \leadsto & (\mathtt B,z-1,2) & \text{if} & z \geqslant 2 \\
    (\mathtt B,z,s) & \overset{1} \leadsto & (\mathtt P,z+1,s) \\
    & \overset{0} \leadsto & (\mathtt B,z-1,s+1) & \text{if} & z \geqslant 2 \\
    (\mathtt P,z,s) & \overset{1} \leadsto & (\mathtt P,z+1,s) \\
    & \overset{1} \leadsto & (\mathtt R,z,s+1) \\
    & \overset{1} \leadsto & (\mathtt M,z,i) & \text{for} & i \in [3,s] \\
    & \overset{0} \leadsto & (\mathtt L,z-1) \\
    (\mathtt L,z) & \overset{1} \leadsto & (\mathtt P,z+1,2) \\
    & \overset{0} \leadsto & (\mathtt L,z-1) & \text{if} & z \geqslant 2 \\
    (\mathtt R,z,s) & \overset{1} \leadsto & (\mathtt P,z+1,s) \\
    & \overset{1} \leadsto & (\mathtt R,z,s+1) \\
    (\mathtt M,z,s) & \overset{1} \leadsto & (\mathtt P,z+1,s) \\
    & \overset{1} \leadsto & (\mathtt M,z,s).
\end{array} \right.$$

The generating function of $\I(010,102)$ can be computed from the succession rule $\Omega_{\{010,102\}}$, with the kernel method. However, this result is easier to achieve by studying only the subset of sequences of $\I(010,102)$ which do not contain the value 1, and this is how we will proceed in Section \ref{sec_proofs_010_102}. Due to this restriction, the construction used in Section \ref{sec_proofs_010_102} ``jumps" over sequences containing the value 1, so they are grayed out in Figure \ref{GT010+102}.

\begin{landscape}
\begin{figure}
\centering
\begin{tikzpicture}[level distance=20mm]
\begin{scope}[scale=0.83]
\tikzstyle{level 1}=[sibling distance=0]
\node{\large $\varepsilon$}
child {node{0}
    child {node{00}
        child {node{000}
            child {node{0000}
                child {node{00000}}
                child[xshift = 14mm, yshift = 20mm] {node{\gray{0001}}
                    child {node{00002}}
                    child[xshift = 14mm, yshift = 20mm] {node{\gray{0011}}
                        child[xshift = 14mm, yshift = 20mm] {node{\gray{0111}}}
                        child {node{00022}}
                    }
                }
            }
            child[xshift = 56mm, yshift = 20mm] {node{\gray{001}}
                child {node{0002}
                    child[xshift = -14mm] {node{00003}}
                    child {node{\gray{00021}}}
                    child[xshift = 14mm, yshift = 20mm] {node{\gray{0012}}
                        child{node{00023}}
                        child[xshift = 14mm, yshift = 20mm] {node{\gray{0112}}}
                    }
                }
                child[xshift = 42mm, yshift = 20mm] {node{\gray{011}}
                    child {node{0022}
                        child[xshift = -14mm] {node{00033}}
                        child {node{\gray{00221}}}
                        child[xshift = 14mm] {node{\gray{00212}}}
                        child[xshift = 14mm, yshift = 20mm] {node{\gray{0122}}}
                    }
                }
            }
        }
        child[xshift = 157mm, yshift = 20mm] {node{\gray{01}}
            child {node{002}
                child[xshift = -24mm] {node{0003}
                    child[xshift = -7mm] {node{00004}}
                    child[xshift = 7mm] {node{\gray{00031}}}
                    child[xshift = 21mm, yshift = 20mm] {node{\gray{0013}}
                        child{node{00024}}
                        child[xshift = 14mm, yshift = 20mm] {node{\gray{0113}}}
                    }
                }
                child[xshift = 24mm] {node{\gray{0021}}
                    child[xshift = -14mm] {node{00032}}
                    child {node{\gray{00211}}}
                }
                child[xshift = 38mm, yshift = 20mm] {node{\gray{012}}
                    child {node{0023}
                        child {node{00034}}
                        child[xshift = 14mm] {node{\gray{00231}}}
                        child[xshift = 14mm, yshift = 20mm] {node{\gray{0123}}}
                    }
                }
            }
        }
    }
};
\end{scope}
\end{tikzpicture}
\caption{Part of a generating tree for $\I(010,102)$. Sequences containing the value 1 are colored in gray.}  
\label{GT010+102}
\end{figure}
\end{landscape}

For the proof of Theorem \ref{thmGF010_102} to come in Section \ref{sec_proofs_010_102}, it is not necessary to prove our claim that the rule $\Omega_{\{010,102\}}$ describes a generating tree isomorphic to that of Figure \ref{GT010+102}. Nevertheless, to help convince the reader of the validity of our claim, we now briefly explain the correspondence between inversion sequences and the labels appearing in this rule. In the following descriptions, the letters $\mu$ and $\nu$ denote nonempty sequences of values greater than or equal to 2.
\begin{itemize}
    \item The letter $\mathtt C$ marks \emph{constant} sequences of the form $(0)^a$ with $a \geqslant 0$.
    \item The letter $\mathtt B$ marks \emph{binary} sequences of the form $(0)^a \cdot (1)^b$ with $a,b > 0$.
    \item The letter $\mathtt P$ marks \emph{pushed} sequences of the form $(0, \dots, 0) \cdot \mu$. These sequences are named after the $\push$ function of Section \ref{sec_proofs_010_102}, and they correspond to the set $\mathfrak B$ in Section \ref{sec_proofs_010_102}.
    \item The letter $\mathtt L$ marks sequences having 1-entries on the \emph{left}, of the form $(0, \dots, 0, 1 \dots, 1) \cdot \mu$.
    \item The letter $\mathtt R$ marks sequences having 1-entries on the \emph{right}, of the form $(0, \dots, 0) \cdot \mu \cdot (1, \dots, 1)$.
    \item The letter $\mathtt M$ marks sequences having 1-entries in the \emph{middle}, of the form $(0, \dots, 0) \cdot \mu \cdot (1, \dots, 1) \cdot \nu$.
    \item The parameter $z$ counts the \emph{zeros} of an inversion sequence.
    \item The parameter $s$ counts the \emph{active sites} of an inversion sequence, for some definition of active sites that is similar, albeit slightly different, from that of Section \ref{sec_proofs_010_102}.
\end{itemize}

\subsection{Enumerating \texorpdfstring{$\I(010, 102)$}{I(010, 102)}} \label{sec_proofs_010_102}

For all $n \geqslant 0$, let $\mathfrak A_n = \{ \sigma \in \I_n(010,102) \; : \; 1 \notin \Vals(\sigma)\}$ be the set of inversion sequences of size $n$ which avoid 010 and 102, and do not contain the value 1. Let $\mathfrak A = \coprod_{n \geqslant 0} \mathfrak A_n$.
For all $\sigma \in \I$, let $\push(\sigma) = 0 \cdot (\sigma_i + \chi(\sigma_i > 0))_{i \in [0,n]}$ be the sequence obtained by inserting a zero at the beginning of $\sigma$, and adding 1 to the value of all nonzero entries.

\begin{prop} \label{propImgPush}
    The image $\push(\I(010,102))$ is the set $\mathfrak A \backslash \{\varepsilon\}$.
\end{prop}
\begin{proof}
Let $\sigma \in \I(010,102)$, and let $\tau = \push(\sigma)$. By construction, $\tau$ is an inversion sequence, $\tau$ does not contain the value $1$, and $|\tau| > |\sigma|$ so $\tau$ cannot be empty. Let us now show that $\tau$ avoids 010 and 102. The sequence $(\tau_i)_{i \in [2,n+1]}$ is order-isomorphic to $\sigma$, which avoids 010 and 102. This means that any occurrence of either pattern in $\tau$ must involve the first entry of $\tau$. Since $\tau_1 = 0$, it cannot be part of an occurrence of 102. If $\tau$ contained an occurrence $(\tau_1, \tau_i, \tau_j)$ of 010 for some $i < j$, then $\sigma$ would contain a zero $\sigma_{j-1}$ after a nonzero value $\sigma_{i-1}$, which is impossible by Proposition \ref{prop010}.

Let $\tau \in \mathfrak A \backslash \{\epsilon\}$, and let $\sigma$ be the sequence obtained by removing the first zero of $\tau$ and subtracting 1 from the value of all nonzero entries. Then $\sigma \in \I(010,102)$, and $\push(\sigma) = \tau$. \qedhere
\end{proof}

\begin{corol} \label{corolPushBijection}
    For all $n \geqslant 0$, $|\I_n(010, 102)| = |\mathfrak A_{n+1}|$.
\end{corol}
\begin{proof}
    Let $n \geqslant 0$. The function $\push$ is injective, so it describes a bijection between $\I_n(010, 102)$ and $\push(\I_n(010,102))$. By construction, the image under $\push$ of a sequence of size $n$ is a sequence of size $n+1$, so it follows from Proposition \ref{propImgPush} that $\push(\I_n(010,102)) = \mathfrak A_{n+1}$. \qedhere
\end{proof} 

For all $n \geqslant 0$ and $\sigma \in \I_n(010,102)$, let
$$\Sites(\sigma) = \{i \in [2, n+1] \; : \; (\sigma_j)_{j \in [1,i-1]} \cdot 1 \cdot (\sigma_j)_{j \in [i,n]} \in \I_{n+1}(010,102)\}$$
be the set of positions, called \textit{active sites}, where inserting the value 1 in $\sigma$ does not create an occurrence of 010 or 102. Note that we do not allow the value 1 to be inserted at the beginning of $\sigma$ (i.e. at position 1), since the result would not be an inversion sequence.

If $\sigma$ already contains the value 1, then a new entry of value 1 can only be inserted next to the current entries of value 1, by Proposition \ref{prop010}, and we have little interest in the active sites of $\sigma$ in that case. We next characterize the active sites of $\sigma$ when $\sigma$ does not contain the value 1. Figure \ref{Fig_active_sites} represents the position of such an active site, as described in Proposition \ref{propSitesA} below.

\begin{prop} \label{propSitesA}
    Let $n \geqslant 1$, $\sigma \in \mathfrak A_n$, and $z = |\Zeros(\sigma)|$. Then
     $$\Sites(\sigma) = \{i \in [z+1,n+1] \; : \; \min((\sigma_j)_{j \in [z+1,i-1]}) \geqslant \max((\sigma_j)_{j \in [i,n]})\}.$$
\end{prop}

\begin{proof} The value 1 cannot be inserted at positions $[2, z]$, since it would create a subsequence $(0,1,0)$. Inserting the value 1 at any other position (i.e. in $[z+1,n+1])$ cannot create an occurrence of 010, since there are no zeros to the right and no other entry of value $1$ in $\sigma$. It follows that $\Sites(\sigma)$ is the subset of $[z+1, n+1]$ of positions where inserting the value 1 does not create an occurrence of 102. Note that all values of $(\sigma_i)_{i \in [z+1, n]}$ are greater than 1, so the new entry of value 1 could only take the role of the minimal value in an occurrence of 102. In other words, $i \in [z+1, n+1]$ is an active site of $\sigma$ if and only if no value of $(\sigma_j)_{j\in [z+1, i-1]}$ is less than a value of $(\sigma_j)_{j \in [i,n]}$. \qedhere
\end{proof}

\begin{figure}[ht]
\centering
\begin{tikzpicture}
    \draw[fill = lightgray] (0,0) -- (0,1) -- (3,1) -- (3,0) -- (0,0);
    \draw[fill = lightgray] (3,3) -- (3,4) -- (6,4) -- (6,3) -- (3,3);
    \draw[fill = lightgray] (6,2) -- (9,2) -- (9,3) -- (6,3) -- (6,2);
    \filldraw (6, 1.5) circle(4pt);
    \node at (6,1.2) {active site};

    \node[left] at (0,0.5) {0};
    \node[left] at (0,1.5) {1};
    \node[left] at (0,3) {$\geqslant 2$};
    \draw (0,0) -- (9,0);
    \draw (0,1) -- (9,1);
    \draw (0,2) -- (9,2);
    \draw (0,4) -- (9,4);
    \draw (0,0) -- (0,4);
    \draw (9,0) -- (9,4);
    \node[rotate = 90] at (-0.7,2) {\textbf{values}};
    \node at (0,-0.3) {1};
    \node[left] at (3,-0.3) {$z$};
    \node[right] at (3,-0.3) {$z+1$};
    \node[left] at (6,-0.3) {$i-1$};
    \node[right] at (6,-0.3) {$i$};
    \node at (9,-0.3) {$n$};
    \node at (4.5,-0.7) {\textbf{positions}};
    \draw[thick, dotted] (3,1) -- (3,3);
    \draw[thick, dotted] (6,0) -- (6,2);
    \draw[thick, loosely dotted] (0.25,-0.3) -- (2.6,-0.3);
    \draw[thick, loosely dotted] (4.1,-0.3) -- (5,-0.3);
    \draw[thick, loosely dotted] (6.5,-0.3) -- (8.8,-0.3);
\end{tikzpicture}
\caption{Illustration of Proposition \ref{propSitesA}.}
\label{Fig_active_sites}
\end{figure}

Let $\mathfrak B = \mathfrak A \backslash \{ (0)^n \; : \; n \geqslant 0\}$ be the subset of non-constant sequences in $\mathfrak A$. Equivalently, $\mathfrak B$ is the set of inversion sequences which avoid 010 and 102, do not contain the value 1, and contain a value greater than 1. For all $n,z,s \geqslant 0$, let $\mathfrak B_n = \{\sigma \in \mathfrak B \; : \; |\sigma| = n\}$, let $\mathfrak B_{n,z} = \{\sigma \in \mathfrak B_n \; : \; |\Zeros(\sigma)| = z\}$, and let $\mathfrak B_{n,z,s} = \{\sigma \in \mathfrak B_{n,z} \; : \; |\Sites(\sigma)| = s\}$. Let also $\mathfrak b_n = |\mathfrak B_n|$, $\mathfrak b_{n,z} = |\mathfrak B_{n,z}|$, and $\mathfrak b_{n,z,s} = |\mathfrak B_{n,z,s}|$.

\begin{lem} \label{lem010+102}
Let $n,z,s \geqslant 0$. If $n \geqslant 3$, $z \in [2,n-1]$, and $s \in [2, n-z+1]$, then
\begin{align*}
    \mathfrak b_{n,z,s} &= \mathfrak b_{n-1,z-1,s}
    + \chi(z+s = n+1)
    + \chi(s=2) \cdot \left ( \sum_{z' \geqslant z} \mathfrak b_{n-1,z'} \right ) \\
    & + \left ( \sum_{k \geqslant 1} \mathfrak b_{n-k-1,z-1,s-k} \right )
    + \chi(s \geqslant 3) \cdot \left ( \sum_{n' \leqslant n-2} \sum_{s' \geqslant s} \mathfrak b_{n',z-1,s'} \right ).
\end{align*}
Otherwise, $\mathfrak b_{n,z,s} = 0$.
\end{lem}
\begin{proof}
    Let $\sigma \in \mathfrak B_{n,z,s}$. By definition, $\sigma$ is not constant, so we have $z < n$. By definition, $\sigma$ cannot contain the value 1, hence $\sigma$ begins by $(0,0)$, and $z \geqslant 2$. This implies that $n \geqslant 3$. By Proposition \ref{propSitesA}, $\Sites(\sigma) \subseteq [z+1,n+1]$, and $z+1$ and $n+1$ are always active sites of $\sigma$, so $s \in [2, n-z+1]$.

    Let $\sigma' = \push^{-1}(\sigma)$ be the sequence obtained by removing the first zero from $\sigma$, and subtracting 1 from every nonzero entry. In particular, $\sigma' \in \I_{n-1}(010,102)$.
    \begin{itemize}
        \item[1] If $\sigma'$ does not contain the value 1, then $\sigma' \in \mathfrak B_{n-1,z-1,s}$.
        \item[2] Otherwise, $\sigma'$ has entries of value 1, and they are all consecutive by Proposition \ref{prop010}.
        \begin{itemize}
            \item[2.1] If $\sigma'_z = 1$, then $\sigma'$ is of the form $(0)^{z-1} \cdot (1, \dots, 1) \cdot \omega$, where $\omega$ is a sequence of values strictly greater than 1, by Proposition \ref{prop010}.
            \begin{itemize}
                \item[2.1.1] If $\omega$ is empty, then $\sigma' = (0)^{z-1} \cdot (1)^{n-z}$, so $\sigma = (0)^{z} \cdot (2)^{n-z}$. In that case, $\Sites(\sigma) = [z+1,n+1]$, so $z+s = n+1$.
                \item[2.1.2] If $\omega$ is not empty, we have $\sigma_{z+1} = 2$ and $\sigma_n > 2$. This implies that inserting the value 1 in $\sigma$ at any position in $[z+2,n]$ would create an occurrence of 102. It follows that $\Sites(\sigma) = \{z+1,n+1\}$, so $s = 2$.
                
                Since the factor of ones of $\sigma'$ is consecutive to its factor of zeros, replacing all the ones by zeros yields a sequence $\tau = (0, \dots, 0) \cdot \omega$, which still avoids 010 and 102. Notice that $\tau$ does not contain the value 1, but contains a value greater than 1, so $\tau \in \mathfrak B_{n-1}$. By construction, the number $z'$ of zeros of $\tau$ is at least $z$.
            \end{itemize}
            \item[2.2] If $\sigma'_z > 1$ and $\sigma'_{n-1} = 1$, let $k$ be the number of entries of value $1$ in $\sigma'$, and let $\tau = (\sigma'_i)_{i \in [1,n-k-1]}$ be the sequence obtained by removing all entries of value 1 from $\sigma'$. In other words, $\sigma' = \tau \cdot (1)^k$, and $\tau = (0)^{z-1} \cdot \omega$, where $\omega$ is a nonempty sequence of values strictly greater than 1. Notice that $\sigma$ ends with $k$ entries of value $2$, and that inserting the value 1 after any of these entries would not create an occurrence of 102. These $k$ entries (or, equivalently, the rightmost $k$ active sites of $\sigma$) are no longer present $\tau$. This means that $\tau$ has $s-k$ actives sites, and so $\tau \in \mathfrak B_{n-k-1, z-1, s-k}$.
            \item[2.3] If $\sigma'_z > 1$ and $\sigma'_{n-1} > 1$, then $\sigma'$ is of the form $(0)^{z-1} \cdot \mu \cdot (1, \dots, 1) \cdot \nu$, where $\mu$ and $\nu$ are two nonempty sequences of values strictly greater than 1, which satisfy $\min(\mu) \geqslant \max(\nu)$ since $\sigma'$ avoids 102. Let $p = z+1+|\mu| \in [z+2, n-1]$ be the position of the leftmost entry of value 2 in $\sigma$. By Proposition \ref{propSitesA}, we have $\{z+1, p, n+1\} \subseteq \Sites(\sigma)$, so $s \geqslant 3$.

            Let $\tau = (0)^{z-1} \cdot \mu \cdot \nu$ be the sequence obtained by removing all entries of value 1 from $\sigma'$.
            Let $n' = |\tau|$, and $s' = |\Sites(\tau)|$, so that $\tau \in \mathfrak B_{n',z-1,s'}$. By construction, $n' \leqslant n-2$.
            Notice that $\Sites(\sigma) \subseteq [z+1,p] \sqcup \{n+1\}$, since inserting the value 1 at any position in $[p+1, n]$ would create an occurrence $(\sigma_{p}, 1, \sigma_n)$ of the pattern 102. This means that all active sites of $\sigma$ are still present in $\tau$, more precisely $\{i-1 \; : \; i \in \Sites(\sigma) \backslash \{n+1\}\} \sqcup \{n'+1\} \subseteq \Sites(\tau)$, so $s' \geqslant s$.
        \end{itemize}
    \end{itemize}

    To complete the proof, it remains to show that each of the five cases above induces a bijection between the set of sequences $\sigma$ and the set of sequences $\sigma'$ or $\tau$. We now describe the inverse bijection in each case.
    \begin{itemize}
        \item[1] Let $n \geqslant 3, z,s \geqslant 2$, $\sigma' \in \mathfrak B_{n-1, z-1, s}$, and $\sigma = \push(\sigma')$. Then $\sigma \in \mathfrak B_{n,z,s}$ and does not contain the value 1.
        \item[2.1.1] Let $n \geqslant 3, z \geqslant 2$, $\sigma' = (0)^{z-1} \cdot (1)^{n-z}$, and  $\sigma = \push(\sigma')$. Then $\sigma = (0)^{z} \cdot (2)^{n-z}$.
        \item[2.1.2] Let $n \geqslant 3, z \geqslant 2, z' \geqslant z$, and $\tau \in \mathfrak B_{n-1, z'}$. Let $\sigma'$ be the sequence obtained by replacing all but the first $z-1$ zeros of $\tau$ by ones, and let $\sigma = \push(\sigma')$. Note that the transformation from $\tau$ to $\sigma'$ cannot create any occurrence of 010 or 102, since there are no zeros to the right of the ones in $\sigma'$. It follows from Proposition \ref{propImgPush} that $\sigma$ still avoids 010 and 102. By construction, $\sigma_{z+1} = \sigma'_z+1 = 2$, and $\sigma_n = \tau_{n-1}+1 > 2$, hence inserting the value 1 in $\sigma$ at any position in $[z+2, n]$ would create an occurrence of 102. It follows that $z+1$ and $n+1$ are the only two active sites of $\sigma$, so $\sigma \in \mathfrak B_{n,z,2}$.
        \item[2.2] Let $n \geqslant 3, z, s \geqslant 2, k \geqslant 1$, and $\tau \in \mathfrak B_{n-k-1, z-1, s-k}$. Let $\sigma' = \tau \cdot (1)^k$, and let $\sigma = \push(\sigma')$. By Proposition \ref{propSitesA}, inserting the value 1 at the end of a sequence cannot create an occurrence of 010 or 102, so $\sigma$ still avoids both patterns. It can be seen that $\sigma$ ends with $k$ entries of value 2, and inserting the value 1 to the right of such an entry cannot create an occurrence of 102, hence all positions in $[n-k+2, n+1]$ are active sites of $\sigma$. Adding to that the $s-k$ active sites from $\tau$, $\sigma$ has a total of $s$ actives sites, so $\sigma \in \mathfrak B_{n,z,s}$.
        \item[2.3] Let $n, s \geqslant 3$, $z \geqslant 2$, $n' \leqslant n-2$, $s' \geqslant s$, and $\tau \in \mathfrak B_{n',z-1,s'}$. Let $p_1 < p_2 < \dots < p_{s'}$ be the active sites of $\tau$. Let $\sigma'$ be the sequence obtained by inserting $n-n'-1$ entries of value 1 at position $p_{s-1}$ in $\tau$, and let $\sigma = \push(\sigma')$. Inserting the value 1 in $\sigma$ at any position in $[p_{s-1}+2, n]$ would create an occurrence of 102, since $\sigma_{p_{s-1}+1} = 2$ and $\sigma_n > 2$. It follows that the active sites of $\sigma$ are $\{p_i + 1 \; : \; i \in [1,s-1]\} \sqcup \{n+1\}$, so $\sigma \in \mathfrak B_{n,z,s}$. \qedhere 
    \end{itemize}
\end{proof}

Let $F(t) = \sum_{n \geqslant 0} |\I_n(010,102)| t^n$ be the ordinary generating function of $\I(010,102)$. Let also $B(t,u,v) = \sum_{n,z,s \geqslant 0} \mathfrak b_{n,z,s} t^n u^z v^s$.
\begin{thm} \label{thmGF010_102}
The function $F(t)$ is algebraic with minimal polynomial
\begin{align*}
&(t^5 -3 t^4 + 4 t^3 - 3 t^2 + t)F(t)^3 + (4 t^4 - 8 t^3 + 6 t^2 - 2t)F(t)^2 \\ & + (-t^4 + 8 t^3 - 11t^2 + 6t - 1) F(t) -2 t^3 + 5 t^2 - 4t + 1.
\end{align*}
\end{thm}
\begin{proof}
By the definition of $\mathfrak B$, we have $|\mathfrak A_n| = \mathfrak b_n + 1$ for all $n \geqslant 0$, so it follows from Corollary \ref{corolPushBijection} that
$$F(t) = \frac{1}{t} \left (B(t,1,1) + \frac{1}{1-t} - 1\right ),$$
and we proceed by computing $B(t,1,1)$.

By Lemma \ref{lem010+102}, we have
\begin{align*}
    B(t,u,v) = \sum_{\substack{n \geqslant 3 \\ z, s \geqslant 2}} & \left [ \mathfrak b_{n-1,z-1,s}
    + \chi(z+s = n+1)
    + \chi(s=2) \cdot \left ( \sum_{z' \geqslant z} \mathfrak b_{n-1,z'} \right ) \right . \\
    & + \left . \left ( \sum_{k \geqslant 1} \mathfrak b_{n-k-1,z-1,s-k} \right )
    + \chi(s \geqslant 3) \cdot \left ( \sum_{n' \leqslant n-2} \sum_{s' \geqslant s} \mathfrak b_{n',z-1,s'} \right ) \right ] t^n u^z v^s.
\end{align*}
We now simplify each of the five terms in this expression to obtain a functional equation characterizing $B(t,u,v)$.
$$\sum_{\substack{n \geqslant 3 \\ z, s \geqslant 2}} \mathfrak b_{n-1,z-1,s} t^n u^z v^s = tu B(t,u,v)$$
$$
    \sum_{\substack{n \geqslant 3 \\ z, s \geqslant 2}} \chi(z+s = n+1) t^n u^z v^s 
    = \sum_{z,s \geqslant 2} t^{-1}(tu)^z (tv)^s
    = \frac{t^3 u^2 v^2}{(1-tu)(1-tv)}
$$
\begin{align*}
    \sum_{\substack{n \geqslant 3 \\ z, s \geqslant 2}} \chi(s=2) \cdot \left ( \sum_{z' \geqslant z} \mathfrak b_{n-1,z'} \right ) t^n u^z v^s 
    &= \sum_{\substack{n \geqslant 2 \\ z \geqslant 2}} \sum_{i \geqslant 0} \mathfrak b_{n,z+i} t^{n+1} u^z v^2 \\
    &= \sum_{\substack{n \geqslant 2 \\ z \geqslant 2}} \sum_{i \in [0, z-2]} \mathfrak b_{n,z} t^{n+1} u^{z-i} v^2 \\
    &= t v^2 \sum_{\substack{n \geqslant 2 \\ z \geqslant 2}} \mathfrak b_{n,z} t^n \cdot \frac{u^2-u^{z+1}}{1-u} \\
    &= t v^2 \left ( \frac{u^2 B(t,1,1) - uB(t,u,1)}{1-u} \right )
\end{align*}
\begin{align*}
    \sum_{\substack{n \geqslant 3 \\ z, s \geqslant 2}} \left ( \sum_{k \geqslant 1} \mathfrak b_{n-k-1,z-1,s-k} \right ) t^n u^z v^s 
    &= \sum_{\substack{z \geqslant 1 \\ n,s \geqslant 2}} \sum_{k \geqslant 1} \mathfrak b_{n,z,s} t^{n+k+1} u^{z+1} v^{s+k} \\
    &= \sum_{\substack{z \geqslant 1 \\ n,s \geqslant 2}}\mathfrak b_{n,z,s} t^{n+1} u^{z+1} v^s \sum_{k \geqslant 1} (tv)^k \\
    &= \sum_{\substack{z \geqslant 1 \\ n,s \geqslant 2}}\mathfrak b_{n,z,s} t^{n+1} u^{z+1} v^s \cdot \frac{tv}{1-tv} \\
    &= t^2uv \left ( \frac{B(t,u,v)}{1-tv} \right )
\end{align*}
\begin{align*}
    \sum_{\substack{n \geqslant 3 \\ z, s \geqslant 2}} \chi(s \geqslant 3) \cdot \left ( \sum_{n' \leqslant n-2} \sum_{s' \geqslant s} \mathfrak b_{n',z-1,s'} \right ) t^n u^z v^s 
    & = \sum_{\substack{n, s \geqslant 3 \\ z \geqslant 1}} \sum_{\substack{i \geqslant 2 \\ j \geqslant 0}} \mathfrak b_{n-i,z,s + j} t^n u^{z+1} v^s \\
    & = \sum_{\substack{n, s \geqslant 3 \\ z \geqslant 1}} \mathfrak b_{n,z,s} \sum_{\substack{i \geqslant 2 \\ j \in [0, s-3]}}  t^{n+i} u^{z+1} v^{s-j} \\
    & = \sum_{\substack{n, s \geqslant 3 \\ z \geqslant 1}} \mathfrak b_{n,z,s} t^n u^{z+1} \cdot \frac{t^2}{1-t} \cdot \frac{v^3-v^{s+1}}{1-v} \\
    & = t^2 u \left (\frac{v^3 B(t,u,1) - vB(t,u,v)}{(1-t)(1-v)} \right )
\end{align*}
In summary, the formula of Lemma \ref{lem010+102} is equivalent to the functional equation
\begin{align*}
    B(t,u,v) &= tu B(t,u,v) + \frac{t^3 u^2 v^2}{(1-tu)(1-tv)} + t v^2 \left ( \frac{u^2 B(t,1,1) - uB(t,u,1)}{1-u} \right )\\
    &+ t^2uv \left ( \frac{B(t,u,v)}{1-tv} \right ) + t^2 u \left (\frac{v^3 B(t,u,1) - vB(t,u,v)}{(1-t)(1-v)} \right ).
\end{align*}
This can be rewritten as
\begin{equation} \label{eqBKernel1}
    K_1(t,u,v) B(t,u,v) = K_2(t,u,v) B(t,u,1) + \frac{t u^2 v^2}{1-u} B(t,1,1) + \frac{t^3 u^2 v^2}{(1-tu)(1-tv)},
\end{equation}
where $K_1$ and $K_2$ are rational functions of $(t,u,v)$:
\begin{align*}
    K_1(t,u,v) &= \frac{- t^3 u v^2 - t^2 v^2 + t^2 u + t^2 v + t u v + t v^2 - t u - t - v + 1}{(1-t) (1-v) (1- tv)}, \\[\medskipamount]
    K_2(t,u,v) &= \frac{tuv^2 (-tuv + t + v - 1)}{(1-t)(1-u)(1-v)}.
\end{align*}
We apply the kernel method. The kernel in \eqref{eqBKernel1} is $K_1(t,u,v)$. This kernel can be cancelled by substituting $v$ for a formal power series $V(t,u)$ which is algebraic with minimal polynomial
$$(t^3 u + t^2 - t) V(t,u)^2 + (- t^2 - tu + 1) V(t,u) - t^2 u + t u + t - 1.$$
An expression for $V(t,u)$ in terms of radicals is
$$V(t,u) = \frac{t^{2}+t u-1+\sqrt{4 t^{5} u^{2}-4 t^{4} u^{2}+t^{4}-2 t^{3} u+t^{2} u^{2}-4 t^{3}+4 t^{2} u+6 t^{2}-2 t u-4 t+1}}{2 t (t^{2} u+t-1)}.$$
Substituting $v$ for $V(t,u)$ in \eqref{eqBKernel1} yields the equation
\begin{equation} \label{eqBKernel2}
K_2(t,u,V(t,u)) B(t,u,1) = - \frac{t u^2 V(t,u)^2}{1-u} B(t,1,1) - \frac{t^3 u^2 V(t,u)^2}{(1-tu)(1-t V(t,u))}.
\end{equation}
Once again, we use the kernel method. The kernel in \eqref{eqBKernel2} is $K_2(t,u,V(t,u))$, and it can be cancelled by substituting $u$ for a formal power series $U(t)$ which is algebraic with minimal polynomial
$$t^2 U(t)^3 - 2 t U(t)^2 + (- t^3 + 2 t^2 + 1) U(t) - t^2 + t - 1.$$
Substituting $u$ for $U(t)$ in \eqref{eqBKernel2}, we obtain
$$B(t,1,1) = - \frac{(1-U(t)) t^3 U(t)^2 V(t,U(t))^2}{ t U(t)^2 V(t,U(t))^2 (1-tU(t))(1-tV(t,U(t)))}.$$
From this equation, and the minimal polynomials of $U$ and $V$, we can compute the minimal polynomial of $B(t,1,1)$ using elimination\footnote{We use the \emph{eliminate} command in Maple.}:
$$(t^3 - 2t^2 + 2t - 1) B(t,1,1)^3 + (t^3 - t^2 - t)B(t,1,1)^2 + (-t^4 + 2t^3 - 4t^2 + t)B(t,1,1) - t^4.$$
Now that we have the minimal polynomial of $B(t,1,1)$, we use elimination again to compute that of ${F(t) = \frac{1}{t} \left (B(t,1,1) + \frac{1}{1-t} - 1\right )}$:
\begin{align*}
&(t^5 -3 t^4 + 4 t^3 - 3 t^2 + t)F(t)^3 + (4 t^4 - 8 t^3 + 6 t^2 - 2t)F(t)^2 \\ & + (-t^4 + 8 t^3 - 11t^2 + 6t - 1) F(t) -2 t^3 + 5 t^2 - 4t + 1. \qedhere
\end{align*}
\end{proof}

\section{A broader perspective on our construction} \label{secConclusion}
In this section, we do a quick survey of the succession rules describing the generating tree growing on the left for inversion sequences avoiding a pattern of length 3. Our goal is mainly to show that the generating tree growing on the left can be used in most of these cases, rather than to provide new enumerative results like in Sections \ref{201+210} and \ref{010+102}. Consequently, we do not prove the correctness of the succession rules presented in this section. None of these rules are ``better" than previously known constructions, but many are similar (in terms of how many distinct labels they require).

We know from Propositions \ref{PropNecessaryPatterns} and \ref{PropSufficientPatterns} that the generating tree growing on the left can be applied to 10 of the 13 classes of inversion sequences avoiding one pattern of length 3 (that is, for every pattern except 000, 010, and 100).
On page \pageref{Table_Rules}, we present succession rules describing the generating tree growing on the left for 9 of these 10 classes, the exception being the class $\I(120)$. The letters $z$, $p$, and $s$ used in these succession rules correspond to the following statistics on inversion sequences:
\begin{itemize}
    \item $z$ is the number of zeros,
    \item $p$ is the number of consecutive zeros appearing as a prefix, i.e. at the beginning of an inversion sequence,
    \item $s$ is the number of consecutive zeros appearing as a suffix, i.e. at the end of an inversion sequence.
\end{itemize}
Additionally, some cases require distinguishing whether a sequence is constant. In that case, we mark labels corresponding to constant sequences with the letter $\mathtt A$, and non-constant sequences with the letter $\mathtt B$.

\begin{landscape}
\begin{table}
\begin{center}
    \setlength{\tabcolsep}{3pt}
    \renewcommand{\arraystretch}{1.25}
    \small{
    \begin{tabular}{|c|c|c|c|}
    \hline
    Pattern & Succession rule & Comment & OEIS \cite{OEIS} \\
    \hline
    001 &
    $\begin{array}{rclll}
    (\mathtt A,0) \\
    (\mathtt A, z) & \leadsto & (\mathtt A, z+1) \\
    && (\mathtt B, i) & \text{for} & i \in [1,z] \\
    (\mathtt B, z) & \leadsto & (\mathtt B,i) & \text{for} & i \in [1,z]
\end{array}$
    & \begin{minipage}{150pt}
        \strut
        Isomorphic to the generating tree growing on the right. It is easy to find simpler constructions for this class.
        \strut
    \end{minipage}
    & \begin{tabular}{c} A011782 \\ (powers \\ of 2) \end{tabular} \\
    \hline
    011 &
    $\begin{array}{rclll}
        (0) \\
        (z) & \leadsto & (z)^z, \, (z+1)
    \end{array}$
    & \begin{minipage}{150pt}
        \strut
        Isomorphic to the generating tree growing on the right used in \cite{Corteel_Martinez_Savage_Weselcouch_2016}.
        \strut
    \end{minipage}
    & \begin{tabular}{c} A000110 \\ (Bell \\ numbers) \end{tabular} \\
    \hline
    012 &
    $\begin{array}{rclll}
        (\mathtt A,0) \\
        (\mathtt A,z) & \leadsto & (\mathtt A,z+1) \\
        && (\mathtt B,i)^{2^{z-1-i}} & \text{for} & i \in [0,z-1]\\
        (\mathtt B,s) & \leadsto & (\mathtt B,s) \\
        && (\mathtt B, i)^{2^{s-1-i}} & \text{for} & i \in [0,s-1]
    \end{array}$
    & \begin{minipage}{150pt}
        \strut
        The construction of \cite{Corteel_Martinez_Savage_Weselcouch_2016} is much simpler (it can be seen as a generating tree with only two distinct labels).
        \strut
    \end{minipage}

    & \begin{tabular}{c} A001519 \\ (odd-index \\ Fibonacci \\ numbers) \end{tabular}\\
    \hline
    021 &
    $\begin{array}{rclll}
        (0) \\
        (p) & \leadsto & (p+1) \\
        && (i)^{2^{p-i}}  & \text{for} & i \in [1,p]
    \end{array}$
    & \begin{minipage}{150pt}
        \strut
        See \cite{Corteel_Martinez_Savage_Weselcouch_2016} for two bijections between $\I(021)$ and Schröder paths, and a bijection with some binary trees. See also \cite{Lin_Kim_2018} for a bijection with two-colored Dyck paths.
        \strut
    \end{minipage}
    & \begin{tabular}{c} A155069 \\ (large \\ Schröder \\ numbers) \end{tabular} \\
    \hline
    $\begin{array}{c}
         101 \\ 110
    \end{array}$
    & $\begin{array}{rclll}
        (0) \\
        (z) & \leadsto & (z+1) \\
        && (i)^{i}  & \text{for} & i \in [1,z]
    \end{array}$
    & \begin{minipage}{150pt}
        \strut
        This construction is used in \cite{Corteel_Martinez_Savage_Weselcouch_2016} and \cite{Beaton_Bouvel_Guerrini_Rinaldi_2019}.
        \strut
    \end{minipage}
    & A113227 \\
    \hline
    102
    & $\begin{array}{l} \begin{array}{rclll}
        (0,0) \\
        (z,z)  & \leadsto & (z+1,z+1)\\
        && (i,j)^{\binom{z-i-1}{z-j}} & \hspace{23.5pt} \text{for} & i \in [0,z-1], \, j \in [i+1,z]
    \end{array} \\
    \text{if} \; s \neq z : \\
    \begin{array}{rclll}
        (s,z) & \leadsto & (s,j) & \text{for} & j \in [s+1, z+1] \\ 
        && (i,j)^{\sum_{k=s-j}^{z-j}\binom{s-i-1}{k}} & \text{for} & i \in [0,s-1], \, j \in [i+1, z]
    \end{array} \end{array}$
    & \begin{minipage}{150pt}
        \strut
        A different construction with two parameters is used in \cite{Mansour_Shattuck_2015} to compute the generating function of $\I(102)$.
        \strut
    \end{minipage}
    & A200753 \\
    \hline
    $\begin{array}{c}
         201 \\ 210
    \end{array}$
    & $\begin{array}{rclll}
        (0,0) \\
        (p,z) & \leadsto & (p+1,j) & \text{for} & j \in [p+1,z+1] \\
        && (i,j)^{\sum_{k = p-j}^{z-j} \binom{p-i}{k}}  & \text{for} & i \in [1,p], \, j \in [i,z]
    \end{array}$
    & \begin{minipage}{150pt}
        \strut
        Two other constructions for this class, involving two parameters each, can be found in \cite{Corteel_Martinez_Savage_Weselcouch_2016} and \cite{Mansour_Shattuck_2015}.
        \strut
    \end{minipage}
    & A263777 \\
    \hline
    \end{tabular}}
\end{center}
\label{Table_Rules}
\end{table}
\end{landscape}

For inversion sequences avoiding the patterns 101, 102, or 110, the generating tree growing on the left can be described with fewer distinct labels than the more naive generating tree growing on the right. For the other 7 patterns (including 120), they require a similar number of labels at each level.

The class $\I(120)$ is not included in the table of page \pageref{Table_Rules}, because the generating tree growing on the left for $\I(120)$ cannot be described by a succession rule involving only one or two parameters as in the other cases. We can label each inversion sequence by the sequence of lengths of its (maximal) factors of zeros, and obtain the unpalatable succession rule below.
$$\Omega_{120} =\left \{ \begin{array}{rclll}
    (0) \\
    (\ell_0, \dots,  \ell_k) & \leadsto & (q_0, q_1, \dots, q_r)^{\binom{\ell_0+1-\sum_{j=0}^r q_j}{r}} \\
    && \text{for} \quad r \geqslant 0, \, q_0, \dots, q_r \in [1, \ell_0+1] \\[4pt]
    && (\ell_0+1, \ell_1, \dots, \ell_{i-1}, q_0, \dots, q_r)^{\binom{\ell_i+1-\sum_{j=0}^r q_j}{r+1}} \\ 
    && \text{for} \quad i \in [1,k], \, r \geqslant 0, \, q_0, \dots, q_r \in [1, \ell_i]
\end{array} \right .$$

The number of distinct labels appearing at level $n$ of the tree described by the rule $\Omega_{120}$ is exponential in $n$. In this case, the approach of \cite{Mansour_Shattuck_2015} is much better since it can be used to compute the enumeration sequence of $\I(120)$ in polynomial time (this is sequence A263778 in \cite{OEIS}). We present some more serviceable succession rules for 120-avoiding inversion sequences ($\Omega'_{120}$ and $\Omega''_{120}$) towards the end of this section.

We now show that even for the classes $\I(000)$, $\I(010)$, and $\I(100)$, it is possible to modify the generating tree growing on the left to obtain some succession rules, in which some of the labels correspond to objects that are not counted (such objects are referred to as ``phantom objects" in \cite{Testart2024}). The enumeration sequences of these three classes can be found in \cite{OEIS} as sequences A000111, A263780, and A263779.

For inversion sequences avoiding the pattern 000, we can use the generating tree growing on the left by allowing sequences to contain more than two zeros. In other words, only occurrences of the pattern 000 involving nonzero values of an inversion sequence are forbidden. This yields a simple succession rule in which sequences are labelled by their number of zeros $z$.
$$\Omega_{000} = \left \{ \begin{array}{rclll}
    (0) \\
    (z) & \leadsto & (z-1)^{\binom{z}{2}}, (z)^z, (z+1)
\end{array} \right .$$
This rule can be used to enumerate 000-avoiding inversion sequences, by only counting nodes of the tree labelled $(0)$, $(1)$, or $(2)$. Note that nodes labelled $(z)$ for $z \geqslant 3$ are still a necessary part of this construction, because some of their descendants have a label whose value is less than 3.
This approach is easy to generalize to inversion sequences avoiding a constant pattern of any length $k \geqslant 2$. We obtain the following succession rule, where only nodes labelled $(z)$ for $z < k$ are to be counted.
$$\Omega_{0^k} = \left \{ \begin{array}{rclll}
    (0) \\
    (z) & \leadsto & (z+1-i)^{\binom{z}{i}} & \text{for} & i \in [0, k-1]
\end{array} \right .$$
Inversion sequences avoiding a constant pattern of any length have been studied before, and their exponential generating function is given in \cite{Hong_Li_2022}. We have not tried to compute generating functions from the rule $\Omega_{0^k}$.

Similarly to the 000 case, we can define the generating tree growing on the left for inversion sequences avoiding occurrences of the pattern 100 which use only nonzero values. This tree can be labelled by pairs of integers $(p, z)$, where $p$ is the length of the prefix of zeros, and $z$ is the number of zeros of the corresponding inversion sequence. It is described by the succession rule below.
$$\Omega_{100} = \left \{\begin{array}{rclll}
    (0,0) \\
    (p,z) & \leadsto & (p+1,z+1) \\
    && (p+1,z)^{z-p} \\
    && (i,j)^{\binom{p-i}{z-j}} & \text{for} & i \in [1,p], \, j \in [i,z] \\
    && (i,j)^{\binom{p-i}{z-j-1}(z-p)} & \text{for} & i \in [1,p], \, j \in [i,z-1]
\end{array} \right .$$
In this tree, inversion sequences which avoid 100 are those that have at most one zero outside of their prefix, and they correspond to nodes labelled $(p, z)$ for $z-p \in \{0,1\}$.

More generally, if $\rho$ is a pattern such that the generating tree growing on the left does not restrict to a tree for $\I(\rho)$ (such patterns are characterized in Proposition \ref{PropSufficientPatterns}), it is still possible to define the generating tree growing on the left for inversion sequences avoiding occurrences of $\rho$ which use only nonzero values. However, in order to enumerate $\I(\rho)$ from such a tree, one must be able to determine which nodes correspond to inversion sequences containing $\rho$, as they must not be counted. Consequently, the labels may need to encode more information than the strict minimum necessary to describe the shape of the tree.

By applying this idea to the pattern 010, we can obtain a generating tree having similar labels to those of the rule $\Omega_{120}$ (to be precise, each inversion sequence is labelled by the number of zeros in its prefix, and the multiset of lengths of the other factors of zeros). The number of distinct labels at level $n$ of this tree is exponential in $\sqrt n$, so this does not provide an efficient way of computing $|\I_n(010)|$. Nevertheless, it is possible to improve this construction to obtain a tree which only has a quadratic number of distinct labels at level $n$, by using the idea of \emph{commitments} from \cite{Pantone201+210}. That tree is described by the succession rule below, where the terms in brackets are unsigned Stirling numbers of the first kind, and only nodes labelled $(p,0)$ must be counted to enumerate $\I(010)$.
$$\Omega_{010} = \left \{\begin{array}{rclll}
    (0,0) \\
    (p,c) & \leadsto & (p+1,c) \\
    && (p+1,c-1) & \text{if} & c > 0 \\
    && (p+1-\ell,c+k)^{\binom{c+k}{k} \Ustirling{\ell}{\ell-k}} & \text{for} & 0 \leqslant k < \ell \leqslant p
\end{array} \right .$$
With this rule, the enumeration sequence of $\I(010)$ can be computed with a similar efficiency to the method previously used in \cite{Testart2024}.

Lastly, we remark that the same approach using commitments also yields a simpler succession rule $\Omega'_{120}$ for inversion sequences avoiding 120.
$$\Omega'_{120} = \left \{\begin{array}{rcll}
    (0,0) \\
    (p,c) & \leadsto & (p+1,c) \\
    && (p+1,c-1) & \text{if} \quad c > 0 \\
    && (p+1-\ell,k)^{a_{\ell, k}} & \text{if} \quad c = 0, \quad\text{for} \quad 0 \leqslant k < \ell \leqslant p
\end{array} \right .$$
The numbers $(a_{\ell,k})_{\ell,k \geqslant 0}$ appearing in the rule $\Omega'_{120}$ count some constrained 120-avoiding words, and their generating function is
$$\sum_{\ell,k \geqslant 0} a_{\ell,k} t^\ell u^k = \frac{2t - 1 +\sqrt{4 t^2 u + 4t^2 - 4tu - 4t + 1}}{2 u (t-1)}.$$
The rule $\Omega'_{120}$ has a nicer form than $\Omega_{010}$, making it possible to skip the commitments. This yields the rule $\Omega''_{120}$ below, which only has one parameter, although it requires allowing arbitrarily long jumps. As a consequence, some nodes of the corresponding ``tree" have an infinite number of children. This construction of $\I(120)$ is actually quite similar\footnote{The construction of $\I(120)$ from \cite{Mansour_Shattuck_2015} works by splitting an inversion sequence into factors starting at each left-to-right maximum, and building each factor from left to right. Here, we essentially build these same factors from right to left.} to the one employed in \cite{Mansour_Shattuck_2015}.
$$\Omega''_{120} = \left \{\begin{array}{rcll}
    (0) \\
    (p) & \overset{1} \leadsto & (p+1) \\
    & \overset{k} \leadsto& (p+k-\ell)^{b_{\ell,k}} & \text{for} \quad 1 \leqslant \ell \leqslant p, \quad k \geq 1
\end{array} \right .$$
The numbers $(b_{\ell,k})_{\ell,k \geqslant 0}$ appearing in the rule $\Omega''_{120}$ count another family of constrained 120-avoiding words, and their generating function is
$$\sum_{\ell,k \geqslant 0} b_{\ell,k} t^\ell u^k = \frac{-2tu + 2t + u -1 + \sqrt{-4 t^2 u + 4t^2 + 4tu + u^2 - 4t -2u + 1}}{2(t-1)}.$$

In conclusion, up to allowing some small modifications, a succession rule involving at most two integer parameters can be found to describe a kind of generating tree growing on the left for every class of inversion sequences avoiding a pattern of length 3. This shows that the construction is quite robust, and it could be further explored in a larger setting; for instance, to study longer patterns, sets of patterns, or different definitions of patterns. We believe this construction could also be applied to \emph{restricted growth functions} \cite{RGF_patterns_2018}, since they form a subtree of the generating tree growing on the left for inversion sequences: the parent of a (nonempty) restricted growth function is always a restricted growth function. However, this does not hold for \emph{ascent sequences} \cite{ascent_sequences_2010}, another interesting subset of inversion sequences.

\acknowledgements
We thank Mathilde Bouvel, Emmanuel Jeandel, and the anonymous referees for their comments which improved the presentation of this work.

\printbibliography

\end{document}